\documentclass[10pt,reqno]{amsart}
\usepackage{amsmath,amssymb,amsthm}
\usepackage{graphicx}
\usepackage[all]{xypic}
\usepackage[utf8,nocaptions]{vietnam}
\usepackage{hyperref}
\input xypic
\def\thmsection{section}
\def\thmchangesection{changesection}

\def\thmchangechapter{changechapter}
\def\thmchange{change}
\def\thmplain{plain}
\ifx\theoremnumberstyle\thmplain
  \theoremstyle{break-italic}
  \newtheorem{satz}{Satz}
\else
  \ifx\theoremnumberstyle\thmsection
    \theoremstyle{break-italic}
    \newtheorem{satz}{Satz}[section]
  \else
    \ifx\theoremnumberstyle\thmchange
      \swapnumbers
      \theoremstyle{break-italic}
      \newtheorem{satz}{Satz}
    \else
      \ifx\theoremnumberstyle\thmchangesection
         \swapnumbers
         \theoremstyle{break-italic}
         \newtheorem{satz}{Satz}[section]
      \else
        \ifx\theoremnumberstyle\thmchangechapter
           \swapnumbers
           \theoremstyle{break-italic}
           \newtheorem{satz}{Satz}[chapter]
        \else
          \ifx\theoremnumberstyle\thmsection
             \theoremstyle{break-italic}
             \newtheorem{satz}{Satz}[section]
          \else
            \theoremstyle{break-italic}
            \newtheorem{satz}{Satz}[section]
          \fi
        \fi
      \fi
    \fi
  \fi
\fi

\theoremstyle{break-italic}
\newtheorem{theorem}[satz]{Theorem}
\newtheorem{lemma}[satz]{Lemma}

\newtheorem{corollary}[satz]{Corollary}
\newtheorem{Proposition}[satz]{Proposition}

\newtheorem*{conjecture*}{Conjecture}

\theoremstyle{break-roman}
\newtheorem{definition}[satz]{Definition}

\newtheorem{example}[satz]{Example}

\newtheorem{remark}[satz]{Remark}

\theoremstyle{standard}

\newtheorem*{claim}{Claim}

\theoremstyle{varthm-roman}
\newtheorem*{varthm-roman}{}
\theoremstyle{varthm-italic}
\newtheorem*{varthm-italic}{}
\theoremstyle{varthm-roman-break}
\newtheorem*{varthm-roman-break}{}
\theoremstyle{varthm-italic-break}
\newtheorem*{varthm-italic-break}{}
\theoremstyle{varthm-roman-no-punctuation}
\newtheorem{varthm-roman-no-punctuation-numbered}[satz]{}
\theoremstyle{varthm-italic-no-punctuation}
\newtheorem{varthm-italic-no-punctuation-numbered}[satz]{}

\newenvironment{varthm-roman-numbered}[1]{
  \begin{varthm-roman-no-punctuation-numbered}
    \mbox{\rm\textbf{#1}}
  }{\end{varthm-roman-no-punctuation-numbered}}
\newenvironment{varthm-italic-numbered}[1]{
  \begin{varthm-italic-no-punctuation-numbered}
    \mbox{\rm\textbf{#1}}
  }{\end{varthm-italic-no-punctuation-numbered}}
\newenvironment{varthm-roman-break-numbered}[1]{
  \begin{varthm-roman-no-punctuation-numbered}
    \mbox{\rm\textbf{#1}\newline}
  }{\end{varthm-roman-no-punctuation-numbered}}
\newenvironment{varthm-italic-break-numbered}[1]{
  \begin{varthm-italic-no-punctuation-numbered}
    \mbox{\rm\textbf{#1}}\newline
  }{\end{varthm-italic-no-punctuation-numbered}}
\numberwithin{equation}{section}

\def\ex{\begin{example}
  }
  \def\eex{\end{example}}
\def\thr{\begin{theorem}}
\def\ethr{\end{theorem}}
\def\pro{\begin{Proposition}}
\def\epro{\end{Proposition}}
\def\coro{\begin{corollary}}
\def\ecoro{\end{corollary}}
\def\df{\begin{definition}}
\def\edf{\end{definition}}
\def\lm{\begin{lemma}}
\def\elm{\end{lemma}}
\def\pf{\begin{proof}}
\def\epf{\end{proof}}
\def\problem{\begin{problem}}
\def\eproblem{\end{problem}}
\def\dlim{\displaystyle\lim}

\def\it{\begin{itemize}}
\def\hit{\end{itemize}}
\def\rem{\begin{remark}}
\def\erem{\end{remark}}

\def\cla{\begin{claim}}
\def\ecla{\end{claim}}
\def\dlim{\displaystyle\lim}

\newcommand{\seq}[1]{\left<#1\right>}

\begin{document}
\title[On directional second-order   tangent sets of  analytic sets]{On directional second-order tangent sets of analytic sets and applications in optimization}
\author{Le Cong Trinh}

\date{\today}
\begin{abstract}
In this paper we study directional second-order tangent sets of real and complex analytic sets. For an analytic set $X\subseteq \mathbb K^n$ and a nonzero tangent direction $u\in T_0X$, we compare the geometric directional second-order tangent set $T^2_{0,u}X$, defined through second-order expansions of analytic curves in $X$, with the algebraic directional second-order tangent set $T^{2,a}_{0,u}X$, defined by the initial forms of the equations of $X$.

We first prove the general inclusion
$T^2_{0,u}X\subseteq T^{2,a}_{0,u}X$
and exhibit explicit real and complex analytic examples showing that this inclusion can be strict. These examples show that algebraically admissible second-order coefficients need not be geometrically realizable by analytic
curves in $X$.

To address this gap, we reformulate the equality \(T^2_{0,u}X=T^{2,a}_{0,u}X\) as a realizability problem: the two sets coincide whenever every algebraically admissible second-order coefficient is realized by an analytic curve in $X$ with prescribed first two terms. We establish this realizability property for several important classes of analytic sets, including smooth analytic germs, homogeneous analytic cones, hypersurfaces with nondegenerate tangent directions, and nondegenerate analytic complete intersections.

As an application, we derive second-order necessary and sufficient optimality conditions for $C^2$ optimization problems on closed sets. In the analytic setting, whenever the above equality holds, the geometric directional second-order tangent sets appearing in these conditions may be replaced by their algebraic counterparts, so that the second-order tests become explicitly computable from the defining equations of the feasible set.
\end{abstract}
\maketitle

\section{Introduction}

Second-order tangent constructions are fundamental tools in variational analysis, optimization, and nonsmooth geometry. They arise naturally in the study of second-order optimality conditions, quadratic growth, stability, and sensitivity analysis for constrained problems, especially when the feasible set is singular and cannot be treated by classical smooth manifold methods. General background on second-order tangent sets and their role in optimization may be found, for example, in the survey of Giorgi, Jiménez, and Novo \cite{GiJiNo10} and in the monograph of Khan, Tammer, and Z\u{a}linescu \cite{KTZ15}.

For a closed set $X\subseteq \mathbb K^n$, the first-order tangent cone $T_0X$ describes feasible first-order directions at the reference point. However, first-order information alone is not sufficient to capture second-order behavior, nor to formulate second-order conditions for local optimality. This motivates the study of directional second-order tangent sets, which encode second-order corrections along a fixed tangent direction $u\in T_0X$.

In the analytic setting, two notions are particularly natural. The geometric directional second-order tangent set $T^2_{0,u}X$ is defined through second-order expansions of analytic curves contained in $X$, and thus reflects
genuine geometric realizability. By contrast, the algebraic directional second-order tangent set $T^{2,a}_{0,u}X$ is defined by the initial forms of the local equations of $X$ and is therefore more explicitly computable. The main purpose of this paper is to clarify the relationship between these two sets.

At the level of first-order tangent cones, the situation is classical. In the complex analytic setting, geometric and algebraic tangent cones coincide by Whitney's theorem \cite{Wh65a, Wh65b}, whereas in the real analytic setting only inclusion holds in general. For directional second-order tangent sets, however, the relationship is more delicate. The geometric set is determined by actual second-order realizability, while the algebraic set is defined by affine conditions
involving the initial homogeneous parts of the defining equations and their next homogeneous terms. There is therefore no a priori reason for the two notions to coincide.

Our first contribution is to prove that for every analytic set $X\subseteq \mathbb K^n$ and every nonzero tangent direction $u\in T_0X$, one always has
\[
T^2_{0,u}X\subseteq T^{2,a}_{0,u}X.
\]
We then construct explicit real and complex analytic examples showing that this inclusion can be strict. These examples reveal that there may exist algebraically admissible second-order coefficients which cannot be realized by
analytic curves in $X$.

A key role in the paper is played by the analytic-curve characterization of the geometric directional second-order tangent set. Namely, a vector $w\in \mathbb K^n$ belongs to $T^2_{0,u}X$ if and only if there exist
$\varepsilon>0$ and an analytic curve $\gamma:(0,\varepsilon)\to X$ such that 
\[
\gamma(t)=tu+\frac12 t^2w+o(t^2)
\qquad (t\to 0).
\]
This characterization shows that $T^2_{0,u}X$ is precisely the set of geometrically realizable second-order coefficients along the direction $u$. Accordingly, the equality
\(
T^2_{0,u}X=T^{2,a}_{0,u}X
\)
is naturally reformulated as a realizability problem: the two sets coincide whenever every algebraically admissible second-order coefficient is realized by an analytic curve in $X$ with the prescribed first two terms.

Using this viewpoint, we establish the equality
\(
T^2_{0,u}X=T^{2,a}_{0,u}X
\)
for several important classes of analytic sets, including smooth analytic germs, homogeneous analytic cones, hypersurfaces with nondegenerate tangent directions, and nondegenerate analytic complete intersections. In these cases,
the geometric second-order information can be recovered from explicit affine algebraic conditions.

The final part of the paper is devoted to applications in optimization.  Building on the framework of Bonnans, Cominetti, and Shapiro~\cite{BCS99} and Rockafellar and Wets~\cite{RoWe97}, we first derive abstract second-order necessary and sufficient optimality conditions for $C^2$ functions on closed sets, stated in terms of the geometric directional second-order tangent sets. We then return to the analytic setting and show that, for the classes covered by our realizability results, these conditions can be checked algebraically by replacing $T^2_{0,u}X$ with $T^{2,a}_{0,u}X$. We also present examples illustrating both the computational advantage of this algebraic formulation and the necessity of distinguishing between geometric and algebraic second-order tangent sets when the two do not coincide.

The paper is organized as follows. Section \ref{sec:preliminaries} recalls the basic notions from real and complex analytic geometry used throughout the paper. Section \ref{sec:geo-alg-tangents} studies geometric and algebraic directional second-order tangent sets and proves the general inclusion $T^2_{0,u}X\subseteq T^{2,a}_{0,u}X$, together with examples showing that the inclusion may be strict. Section \ref{sec:equality} reformulates the equality problem in terms of second-order realizability and establishes it for several classes of analytic sets. Section \ref{sec:optimization} is devoted to applications to second-order optimality conditions.

\section{Preliminaries}\label{sec:preliminaries}\label{sec:preliminaries}

In this section we recall several basic notions from real and complex analytic geometry that will be used throughout the paper. Standard references include Krantz and Parks \cite{KrPa02} for real-analytic functions, Narasimhan \cite{Nara85} for analytic spaces and germs, Chirka \cite{Chirka89} for complex analytic sets, and Whitney \cite{Wh65a, Wh65b} for tangent cones in the complex analytic setting.

\begin{definition}[\cite{KrPa02,Chirka89}] \rm  
Let $U \subseteq \mathbb K^n$ be open. A function $f:U\to \mathbb K$ is called \emph{analytic} if, for every $a\in U$, there exists a neighborhood $V\subseteq U$ of $a$ such that $f$ is represented on $V$ by a convergent power series centered at $a$. In the case $\mathbb K=\mathbb R$, this means that $f$ is real-analytic; in the case $\mathbb K=\mathbb C$, this means that $f$ is holomorphic.

Denote $\mathcal O_{\mathbb K^n}(U)$ the ring of analytic functions $U \to \mathbb K$, where $U \subseteq \mathbb K^n$ is an open subset.
\end{definition}

\begin{definition}[\cite{Nara85,Chirka89}] \rm
A subset $X\subseteq U$ is called an \emph{analytic set} if, for every point $a\in U$, there exists a neighborhood $V\subseteq U$ of $a$ and finitely many analytic functions
\[
f_1,\dots,f_r \in \mathcal O_{\mathbb K^n}(V)
\]
such that
\[
X\cap V=\{x\in V:\ f_1(x)=\cdots=f_r(x)=0\}.
\]
Thus an analytic set is a subset locally defined as the common zero locus of finitely many analytic functions.
\end{definition} 

\begin{definition}[\cite{Nara85,Chirka89}]\rm 
An analytic set $X\subseteq \mathbb K^n$ is called a \emph{smooth analytic submanifold near $0\in X$} if there exists a neighborhood $U$ of $0$ such that $X\cap U$ is an analytic submanifold of $\mathbb K^n$. Equivalently, there exists an integer $d$ and a local analytic coordinate system near $0$ in which
\[
X\cap U \cong \mathbb K^d\times \{0\}
\]
as analytic sets. In particular, near a smooth point, an analytic set is locally analytically equivalent to a linear subspace. 
\end{definition}

\begin{definition}[\cite{Nara85,Chirka89}] \rm 
Two subsets $A,B\subseteq \mathbb K^n$ are said to have the same \emph{germ at $0\in \mathbb K^n$} if there exists a neighborhood $V$ of $0$ such that
\[
A\cap V = B\cap V.
\]
The equivalence class of $A$ under this relation is called the \emph{germ of $A$ at $0$} and is denoted by $(A,0)$. Likewise, two analytic functions defined near $0$ represent the same \emph{germ of analytic function at $0$} if they coincide on some neighborhood of $0$. The ring of germs of analytic functions at $0$ on $\mathbb K^n$ is denoted by $\mathcal O_{\mathbb K^n,0}$.
\end{definition}

\begin{definition}[\cite{Nara85,Chirka89}]\rm 
Let $X\subseteq \mathbb K^n$ be an analytic set with $0\in X$. The \emph{ideal of germs vanishing on $X$ at $0$} is defined by
\[
\mathcal{I}(X,0):=\{\,f\in \mathcal O_{\mathbb K^n,0} : f\vert_{X\cap V}=0
\text{ for some neighborhood }V\text{ of }0\,\}.
\]
Equivalently, $\mathcal{I}(X,0)$ consists of all germs of analytic functions at $0$ that vanish on the germ $(X,0)$.
\end{definition}

\begin{definition}[\cite{GLS07}]\rm An analytic set $X$ in $ \mathbb{K}^n$ is called an \textit{analytic complete intersection at $0$} if $\mathcal{I}(X,0)$ can be generated by exactly $p = \operatorname{codim}_0 X = n-\dim_0 X$ analytic function germs
\[ \mathcal{I}(X,0) = \seq{f_1, \dots, f_p}, \]
and near $0$,
\[ X = \{x \in \mathbb{K}^n : f_1(x) = \cdots = f_p(x) = 0\}. \]
Equivalently, the minimal number of generators of the ideal equals the codimension.
\end{definition}

\begin{definition}[\cite{Nara85,Chirka89}]\rm 
Let $f\in \mathcal O_{\mathbb K^n,0}$, $f\neq 0$. Then $f$ admits a convergent power-series expansion
\[
f = f_m + f_{m+1}+f_{m+2}+\cdots,
\]
where each $f_j$ is a homogeneous polynomial of degree $j$, and $f_m\not\equiv 0$ for some smallest integer $m\ge 0$. The homogeneous polynomial $f_m$ is called the \emph{initial form} of $f$ at $0$ and is denoted by $f^{[\ast]}$. The integer
\[
\operatorname{ord}_0(f):=m
\]
is called the \textit{order of $f$ at $0$}.
\end{definition}

\begin{definition}[\cite{Wh65a, Wh65b, Chirka89}]\rm  Let $X\subseteq \mathbb K^n$ be an analytic set with $0\in X$.
The \textit{geometric tangent cone} $T_0 X$ of $X$ at $0$ is defined by the set of all $v \in \mathbb K^n$ such that there exists a sequence $\{x_k\} \subseteq X \setminus \{0\}$ which converges to $0$ and a sequence of numbers $\{t_k\} \subseteq \mathbb K$ for which the sequence $t_k x_k$ converges to $v$.
Note that, it will not alter the definition if we require that the $t_k$ real and positive \cite[Remark 8.2]{Wh65b}. Hence, 
\[
T_0 X = \{v \in \mathbb K^n \vert \exists \{x_k\} \subseteq X \setminus \{0\}, \{t_k\} \subseteq \mathbb R, t_k>0 \text{ s.t. } x_k \to 0 \text{ and } t_k x_k \to v\}.
\]

\end{definition}

\begin{definition}[\cite{Wh65a, Wh65b, Chirka89}] \rm
Let $X\subseteq \mathbb K^n$ be an analytic set with $0\in X$. The \emph{algebraic tangent cone} of $X$ at $0$, denoted by $C_0(X)$, is the algebraic set defined by the initial forms of all germs in $\mathcal{I}(X,0)$, namely
\[
C_0(X)
:=
\Bigl\{u\in \mathbb K^n:\ f^{[\ast]}(u)=0\quad \forall f\in \mathcal{I}(X,0)\Bigr\}.
\]
Equivalently, if $\operatorname{In}(\mathcal{I}(X,0))$ denotes the ideal generated by the initial forms $f^{[\ast]}$, $f\in \mathcal{I}(X,0)$, then
\[
C_0(X)=V\bigl(\operatorname{In}(\mathcal{I}(X,0))\bigr).
\]
Thus $C_0(X)$ is a homogeneous algebraic set, hence a cone. 
\end{definition}
\begin{remark} \rm 
  Let $X\subseteq \mathbb K^n$ be an analytic set with $0\in X$. It was shown in \cite{Wh65b} (see also in \cite{Chirka89}, \cite{GH}, \cite{LeLu20}, \cite{OW04}) that the geometric tangent cone  $T_0 X $ is always \textit{contained in} the  algebraic tangent cone $C_0 X $ of $X$ at $0$. Moreover, in the case $\mathbb K= \mathbb C$, Whitney showed in \cite{Wh65b} that $T_0 X = C_0 X$. However, in the case where $\mathbb K=\mathbb R$,  $T_0 X $ is in general   different from $C_0 X$ (cf. \cite{OW04}, \cite{LeLu20}). 
\end{remark}

\section{Directional second-order  tangent sets of analytic sets }\label{sec:geo-alg-tangents}
In the following we use $\mathbb K $ to  denote for the field of real numbers $\mathbb R$ or the field of complex numbers $\mathbb C$. 

\subsection{Geometric directional second-order tangent sets}
 
Let $X\subseteq \mathbb K^n$, $p\in X$, and $u\in \mathbb K^n$. 

  \df \label{def:2.3} \rm The \textit{geometric second-order tangent set} $T^2_{p,u} X$ of $X$ \textit{at $p \in \mathbb K^n$} \textit{in the direction} $u\in \mathbb K^n$ is defined by the set of all $w\in \mathbb K^n$ for which there are a sequence of positive real numbers $\{t_k\}$ with $t_k \downarrow 0$ and a sequence $\{w_k\}$ in $\mathbb K^n$ with $w_k \rightarrow w$ such that $p+t_ku+\dfrac{1}{2}t_k^2w_k \in X$ for every $k\in \mathbb N$.
\edf 

There are various types of second-order tangent sets of subsets in $\mathbb K^n$ \cite[Section 4.9]{KTZ15}.  Some equivalent definitions for  $T^2_{p,u} X$ can be found in \cite[Theorem 4.9.4]{KTZ15}. 

In the following we select some properties of the set $T^2_{p,u}X$. 

\begin{Proposition}[{\cite[Theorem 4.9.5]{KTZ15}}] \label{propertyT2} Let $X, X_1, X_2$ be subsets of $\mathbb K^n$; $u \in \mathbb K^n$, and $p \in X \cap X_1 \cap X_2$. Then 
\begin{itemize}
\item[(1)] $T^2_{p,0} X = T_p X$;
\item[(2)] $0\in T^2_{p,u}X $ $\Rightarrow $ $u\in T_pX$;
\item[(3)] $T^2_{p,u}X$ is closed;
\item[(4)] $T^2_{p,u}X \not = \emptyset $ $\Rightarrow $ $u\in T_p X$, hence, $u\not \in T_pX$ $\Rightarrow $ $T^2_{p,u}X = \emptyset$;
\item[(5)] $X_1\subseteq X_2$ $\Rightarrow$ $T^2_{p,u}X_1 \subseteq T^2_{p,u}X_2$;
\item[(6)] $T^2_{p,u} (X_1\cup X_2) = T^2_{p,u}X_1 \cup T^2_{p,u}X_2$.
\end{itemize} 
\end{Proposition}

\textit{In the whole paper, because of the simplicity, we  consider directional second-order tangent sets of an analytic set $X$ at $0 \in X$. }

\begin{remark} \rm $T^2_{0,u}X$ is, in general, not a cone in $\mathbb K^n$. 
\end{remark}
In fact,  let 
\[
X:=\{(x,y)\in \mathbb K^2 : y-x^2=0\},
\qquad 0=(0,0)\in X,
\qquad u:=(1,0).
\]
This is an analytic (indeed algebraic) hypersurface in $\mathbb K^2$.

Let $w=(w_1,w_2)\in \mathbb K^2$.
By definition, $w\in T^2_{0,u}X$ if and only if there exist sequences
$t_k\downarrow 0$ and $w_k=(a_k,b_k)\to (w_1,w_2)$ such that
\[
x_k:=t_k u+\frac12 t_k^2 w_k
=
\left(t_k+\frac12 t_k^2 a_k,\ \frac12 t_k^2 b_k\right)
\in X
\quad\text{for all }k.
\]
Since $x_k\in X$, we must have
\[
\frac12 t_k^2 b_k
=
\left(t_k+\frac12 t_k^2 a_k\right)^2
=
t_k^2+t_k^3 a_k+\frac14 t_k^4 a_k^2.
\]
Dividing by $t_k^2$ and letting $k\to\infty$, we obtain
\[
\frac12 w_2 = 1,
\quad\text{hence}\quad
w_2=2.
\]
Conversely, for any $w_1\in\mathbb K$ with $w_2=2$, one easily constructs sequences
satisfying the definition, so
$$
T^2_{0,u}X=\{(w_1,w_2)\in\mathbb K^2 : w_2=2\}.
$$

This set is an affine hyperplane and is not invariant under scalar multiplication.
For example, $(0,2)\in T^2_{0,u}X$ but $(0,4)=2(0,2)\notin T^2_{0,u}X$.
Hence $T^2_{0,u}X$ is not a cone.

\begin{Proposition}[Analytic-arc characterization of $T^2_{0,u}X$] \label{prop:curve-charac}
Let $X\subseteq \mathbb K^n$  be an analytic set with $0\in X$, and let
$u\in \mathbb K^n\setminus\{0\}$. 
Then for $w\in \mathbb K^n$ the following are equivalent.
\begin{enumerate}
\item[(1)] $w\in T^2_{0,u}X$.
\item[(2)] There exist $\varepsilon>0$ and an analytic curve
$\gamma:(0,\varepsilon)\to X$ (real-analytic if $\mathbb K=\mathbb R$, holomorphic if $\mathbb K=\mathbb C$)
such that
\[
\gamma(t)=t\,u+\frac12 t^2 w+o(t^2)\qquad (t\to 0).
\]
Equivalently,
\[
\lim_{t\to 0}\frac{\gamma(t)-t\,u}{t^2/2}=w.
\]
\end{enumerate}
\end{Proposition}

\begin{proof} Assume there exists an analytic curve $\gamma: (0,\varepsilon)\to X$ such that
\[
\gamma(t)=t\,u+\frac12 t^2 w+o(t^2)\quad (t\to 0).
\]
Choose any sequence $t_k\downarrow 0$ and define
\[
w_k:=\frac{2}{t_k^2}\bigl(\gamma(t_k)-t_k u\bigr).
\]
Then $w_k\to w$ and
\[
t_k u+\frac12 t_k^2 w_k=\gamma(t_k)\in X,
\]
which shows $w\in T^2_{0,u}X$ by definition.

\medskip

Conversely, assume $w\in T^2_{0,u}X$. Then there exist sequences $t_k\downarrow 0$ and $w_k\to w$ such that
\begin{equation}\label{eq:sequence}
t_k u+\frac12 t_k^2 w_k\in X\qquad \forall k.
\end{equation}
Consider the set
\[
A:=\Bigl\{(t,v)\in \mathbb  K\times \mathbb K^n:\ t\neq 0,\ \ t u+\tfrac12 t^2 v\in X\Bigr\}.
\]
Since $X$ is analytic and the mapping
\[
(t,v)\longmapsto t u+\tfrac12 t^2 v
\]
is analytic, the set $A$ is subanalytic (respectively complex-analytic when $\mathbb K=\mathbb C$).
By \eqref{eq:sequence}, $(t_k,w_k)\in A$ and $(t_k,w_k)\to (0,w)$, hence $(0,w)\in\overline{A}$.

By the Curve Selection Lemma for subanalytic (respectively analytic) sets
\cite[Th\'eor\`eme~1]{Lo95},
\cite[\S3]{Mil68},
there exist $\varepsilon>0$ and an analytic mapping
\[
\sigma:(0,\varepsilon)\to A,\qquad \sigma(s)=(t(s),v(s)),
\]
such that
\[
\lim_{s\to 0}(t(s),v(s))=(0,w).
\]
In particular, $t(s)\neq 0$ for all $s\in(0,\varepsilon)$ and $v(s)\to w$ as $s\to 0$.
Define an analytic curve
\[
\tilde\gamma(s):=t(s)\,u+\frac12 t(s)^2 v(s)\in X.
\]

Since $t(s)$ is a nonconstant analytic function with $t(s)\to 0$, there exist
$m\in\mathbb N$ and $a\neq 0$ such that
\[
t(s)=a s^m+o(s^m)\qquad (s\to 0).
\]
After possibly restricting $(0,\varepsilon)$, we may reparametrize by $\tau:=t(s)$
(cf.\ \cite[p.~27]{Mil68}). Writing $s=s(\tau)$ and setting
\[
\gamma(\tau):=\tilde\gamma(s(\tau)),
\]
we obtain an analytic curve $\gamma:(0,\varepsilon')\to X$ satisfying
\[
\gamma(\tau)=\tau u+\frac12 \tau^2\,v(s(\tau)).
\]
Since $v(s)\to w$, we conclude
\[
\gamma(\tau)=\tau u+\frac12 \tau^2 w+o(\tau^2)\qquad (\tau\to 0),
\]
which proves \textup{(2)}.
\end{proof}

\begin{remark}\label{rem:importance-prop-2.6}\rm
Proposition \ref{prop:curve-charac} is fundamental for the rest of the paper. Indeed, the definition
of $T^2_{0,u}X$ in Definition \ref{def:2.3} is given in terms of approximating sequences,
whereas Proposition \ref{prop:curve-charac} shows that membership in $T^2_{0,u}X$ is equivalent to
the existence of an analytic arc in $X$ having prescribed first- and second-order
terms. In this way, the proposition gives a geometric interpretation of the
directional second-order tangent set: it is precisely the set of second-order
coefficients that are realizable by analytic arcs tangent to $u$. 

This characterization plays a decisive role in Section~ \ref{sec:equality}. There, the comparison
between the geometric set $T^2_{0,u}X$ and the algebraic set $T^{2,a}_{0,u}X$
is reformulated as a realizability problem.
\end{remark}

\subsection{Algebraic directional second-order tangent sets of analytic sets}
In this section we study the second-order tangent set  of an analytic set $X$ in some open set $U\subseteq \mathbb K^n$ at $0\in X$ in the direction $u\in \mathbb K^n$.

For an analytic function $f$ in a neighborhood $U$ of $0\in \mathbb K^n$, its \textit{gradient vector } at a point $x^0 \in U$  is defined by 
$$ \nabla f (x^0):= \big(\dfrac{\partial f }{\partial x_1}(x^0), \ldots, \dfrac{\partial f}{\partial x_n}(x^0)\big); $$
its \textit{Hessian matrix} at $x^0 \in U$ is defined by 
$$ \nabla^2 f (x^0):= \big(\dfrac{\partial^2 f }{\partial x_i \partial x_j}(x^0)\big)_{i,j=1,\cdots,n}. $$
More generally, let   $\nabla^{(p)}f(x^0)$ denote the $p$-th order derivative of $f$ at $x^0$, and for any $y\in \mathbb K^n$, let us denote  
$$
 \nabla^{(p)}f(x^0)[y]^p:=\nabla^{p}f(x^0)\underbrace{(y)\ldots (y)}_{ \mbox{$p$ times}},
$$
where $[y]^p$ denote the $p$-times outer product of the vector  $y\in \mathbb K^n$, and $\nabla^{(p)}f(x^0)[y]^p$ the outer product of two arrays $\nabla^{(p)}f(x^0)$ and $[y]^p$. 
For example, 
$$\nabla^{(1)}f(x^0)[y]^1 = \seq{\nabla f(x^0),y}, \quad \nabla^{(2)}f(x^0)[y]^2 = \seq{\nabla^2 f(x^0)y, y}. $$
Then for $h\in \mathbb K^n$ such that $x^0+h \in U$, the Taylor's expansion of $f$ at $x^0$ can be written as
\begin{equation}\label{Taylor}
 f(x^0+h)=\sum_{i=0}^\infty \dfrac{1}{i!}\nabla^{(i)}f(x^0)[h]^i.
\end{equation}

\df \label{def:algebraic-tangent-set}\rm Let $X\subseteq \mathbb K^n$ be an analytic set with  $0 \in X$. The \textit{algebraic second-order tangent set} of $X$ at $0$ \textit{in the direction} $u\in \mathbb K^n$ is defined by 
$$T^{2,a}_{0,u} X:=\Big\{w\in \mathbb K^n \mid \dfrac{1}{2}\seq{\nabla f^{[*]}(u),w} + f^{[*]+1}(u) = 0, \forall f \in \mathcal{I}(X,0)\Big\}.$$ 
Note that, if  $\nabla f^{[*]}(u) \not = 0$, $T^{2,a}_{0,u} X$ is a  hyperplane  in $\mathbb K^n$. 
\edf 

\begin{remark} \rm $T^{2,a}_{0,u}X$ is, in general, not a cone in $\mathbb K^n$. 
\end{remark}
In fact,  let 
\[
X:=\{(x,y)\in \mathbb K^2 : y-x^2=0\},
\qquad 0=(0,0)\in X,
\qquad u:=(1,0).
\]
The ideal $\mathcal{I}(X,0)$ is generated by
\(
f(x,y)=y-x^2.
\)
Its lowest degree homogeneous part is
\(
f^{[*]}(x,y)=y,
\)
and the next homogeneous part is
\(
f^{[*]+1}(x,y)=-x^2.
\)
We compute
\[
\nabla f^{[*]}(x,y)=(0,1),
\qquad
\nabla f^{[*]}(u)=(0,1),
\qquad
f^{[*]+1}(u)=-1.
\]
By definition,
\[
T^{2,a}_{0,u}X
=
\left\{
w\in\mathbb K^2 :
\frac12\langle \nabla f^{[*]}(u),w\rangle + f^{[*]+1}(u)=0
\right\}.
\]
Since $\langle (0,1),(w_1,w_2)\rangle=w_2$, this condition becomes
\[
\frac12 w_2 -1=0,
\quad\text{that is,}\quad
w_2=2.
\]
Therefore
\[
T^{2,a}_{0,u}X=\{(w_1,w_2)\in\mathbb K^2 : w_2=2\}.
\]

As before, this set is not closed under scalar multiplication, hence it is not a cone.

In the following we give an inclusion between these two sets.

\begin{Proposition} \label{inclusion} Let $X$ be an analytic set in an open set $U\subseteq \mathbb K^n$; $0\in X$. Let $u \in \mathbb K^n\setminus \{0\}$. Then 
\begin{equation}\label{equ-inclusion} T^2_{0,u} X \subseteq T^{2,a}_{0,u} X.
\end{equation}
\end{Proposition} 
\begin{proof} We may assume  that $ T^2_{0,u} X \not = \emptyset$, and take any $w \in  T^2_{0,u} X$. Then by the characterization of $T^2_{0,u} X $, there exist  $\{t_k\}\rightarrow 0,$ $  \{x_k\}\subseteq X$ such that $$\dlim_{k\rightarrow \infty } \dfrac{x_k-t_ku}{\frac{1}{2}t_k^2} = w. $$
Denote  $w_k:=\dfrac{x_k-t_ku}{\frac{1}{2}t_k^2}$ which converges to $w$ as $k\rightarrow \infty$. Then 
$$x_k = t_ku + \frac{1}{2}t_k^2w_k.$$ 
Let take any $f\in \mathcal{I}(X,0)$. Denote $m$ be the order of $f$. Then we can write 
$$ f = \sum_{i=m}^\infty f_i, $$
where $f_i$ is the homogeneous component of $f$ whose degree equals to $i$. Then for each $i=m, m+1,\ldots,$ it follows from the Taylor's expansion (\ref{Taylor}) and the homogeneous property of $f_i$  that  
$$f_i(x_k)=t_k^i f_i(u) + \dfrac{1}{2} t_k^{i+1} \seq{\nabla f_i(u),w_k} + \dfrac{1}{8}t_k^{i+2}\seq{\nabla^2 f_i(u)w_k, w_k} + o(t_k^{i+2}).$$
Now we have 
\begin{align}\label{Taylor-f}
f(x_k) & = t_k^m f_m(u) + \dfrac{1}{2} t_k^{m+1} \seq{\nabla f_m(u),w_k} + \dfrac{1}{8}t_k^{m+2}\seq{\nabla^2 f_m(u)w_k, w_k} + \\
& + t_k^{m+1} f_{m+1}(u) + \dfrac{1}{2} t_k^{m+2} \seq{\nabla f_{m+1}(u),w_k} + \dfrac{1}{8}t_k^{m+3}\seq{\nabla^2 f_{m+1}(u)w_k, w_k} + o(t_k^{m+3}).\nonumber
\end{align}
Since $x_k\in X$ we have $f(x_k)=0$. It follows that 
\begin{align*}
0 & =  f_m(u) + \dfrac{1}{2} t_k  \seq{\nabla f_m(u),w_k} + \dfrac{1}{8}t_k^{2}\seq{\nabla^2 f_m(u)w_k, w_k} + \\
& + t_k f_{m+1}(u) + \dfrac{1}{2} t_k^{2} \seq{\nabla f_{m+1}(u),w_k} + \dfrac{1}{8}t_k^{3}\seq{\nabla^2 f_{m+1}(u)w_k, w_k} + o(t_k^{3}).
\end{align*}
Letting $k $ tend to $\infty$, since $t_k \rightarrow 0$, we have $f_m(u)=0$ (i.e. $u\in C_0(X)$), and 
\begin{align*}
0 & =    \dfrac{1}{2} \seq{\nabla f_m(u),w_k} + \dfrac{1}{8}t_k \seq{\nabla^2 f_m(u)w_k, w_k} + \\
& +  f_{m+1}(u) + \dfrac{1}{2} t_k  \seq{\nabla f_{m+1}(u),w_k} + \dfrac{1}{8}t_k^{2}\seq{\nabla^2 f_{m+1}(u)w_k, w_k} + o(t_k^{2}).
\end{align*}
Letting $k$ tend to $\infty$ again, note that $t_k \rightarrow 0$ and $w_k \rightarrow w$, we obtain 
$$ \dfrac{1}{2} \seq{\nabla f_m(u),w}  + f_{m+1}(u) = 0,$$
i.e. $w \in T^{2,a}_{0,u} X$. 
\end{proof}

\subsection{Distinction of second-order tangent sets}\label{subsec:difference}

 In this subsection we present examples showing that the inclusion $T^{2}_{0,u}X \subseteq T^{2,a}_{0,u}X$ can be strict.

Firstly, we consider the case where $X$ is a \textbf{real analytic set}. For example, let us consider
$$
X=\bigl\{(x,y,z)\in\mathbb R^3:\ f(x,y,z)=z^{2}-x^{3}y^{3}=0\bigr\},\qquad 0\in X,
$$
and take \(u\in T_{0}X\setminus\{0\}\).

The tangent cone of $X$ at $0\in X$ is 
$$
T_0 X=\{(x_1,x_2,0) \mid x_1,x_2\in \mathbb R\}.
$$
Hence, we may take $u=(u_{1},u_{2},0)\not = (0,0,0)$.

Now we compute $T^{2,a}_{0,u}X$.   Since $X$ is a hypersurface germ defined by one analytic equation,
and $f = z^2 - x^3 y^3$ is square-free in the local analytic ring, the ideal $ \mathcal{I}(X, 0) $ of analytic germs vanishing on $(X, 0)$ is exactly the principal ideal generated by $f$:
$$\mathcal{I}(X, 0) = \left<f\right> = \{f \cdot g \vert g\in \mathcal{O}_{\mathbb R^3,0}\}.$$ Here
$$
f^{[*]}(x,y,z)=z^{2},\qquad f^{[*]+1}\equiv 0,\qquad \nabla(z^{2})=(0,0,2z).
$$
Since \(u_{3}=0\), we get \(\nabla(z^{2})(u)=(0,0,0)\). Therefore the second-order condition becomes
$$
\frac12\langle 0,w\rangle +0=0,
$$
which is always true. Hence for every \(u=(u_{1},u_{2},0)\neq (0,0,0)\),
$$
T^{2,a}_{0,u}X=\mathbb R^{3}.
$$

Next we compute $T^{2}_{0,u}X$. By definition, \(w\in T^{2}_{0,u}X\) if and only if there exist \(t_{k}\downarrow 0\) and \(w_{k}\to w\) such that
\[
x_{k}:=t_{k}u+\frac12 t_{k}^{2}w_{k}\in X.
\]
Write \(w_{k}=(a_{k},b_{k},c_{k})\). Then
\[
x_{k}=\Bigl(t_{k}u_{1}+\frac12 t_{k}^{2}a_{k},\ \ t_{k}u_{2}+\frac12 t_{k}^{2}b_{k},\ \ \frac12 t_{k}^{2}c_{k}\Bigr),
\]
and the constraint \(x_{k}\in X\) is
\begin{equation}\label{eq:star}
\Bigl(\frac12 t_{k}^{2}c_{k}\Bigr)^{2}
=
\Bigl(t_{k}u_{1}+\frac12 t_{k}^{2}a_{k}\Bigr)^{3}
\Bigl(t_{k}u_{2}+\frac12 t_{k}^{2}b_{k}\Bigr)^{3}.
\end{equation}
The left side of (\ref{eq:star}) satisfies
\[
\Bigl(\frac12 t_{k}^{2}c_{k}\Bigr)^{2}\sim \frac14 t_{k}^{4}c_{k}^{2},
\]
while the right side is at least of order \(t_{k}^{6}\) whenever \(u\neq 0\). Hence \(c_{k}\to 0\), and therefore
\[
w_{3}=0\qquad\text{for every } w\in T^{2}_{0,u}X.
\]
 
We consider four  cases.\\
{\textbf{Case 1:} \(u_{1}u_{2}>0\).}\\
Then the right-hand side of \eqref{eq:star} has leading term \(t_{k}^{6}(u_{1}u_{2})^{3}>0\), and one can satisfy
\eqref{eq:star} by choosing \(c_{k}\) of order \(t_{k}\) (so that \(z_{k}\) is order \(t_{k}^{3}\)), while still having \(c_{k}\to 0\).
The sequences \(a_{k},b_{k}\) can converge arbitrarily. Thus
\[
u_{1}u_{2}>0 \ \Longrightarrow\
T^{2}_{0,u}X=\{(w_{1},w_{2},0): w_{1},w_{2}\in\mathbb R\}.
\]
{\textbf{Case 2:} \(u_{1}u_{2}<0\).}\\
Then the leading coefficient on the right-hand side of \eqref{eq:star} is \((u_{1}u_{2})^{3}<0\), while the left-hand side is a square and hence \(\ge 0\).
For \(k\) large the right-hand side stays negative, so equality is impossible. Hence
\[
u_{1}u_{2}<0 \ \Longrightarrow\ T^{2}_{0,u}X=\varnothing.
\]
{\textbf{Case 3:} \(u_{1}=0,\ u_{2}\neq 0\).}\\
Then the right-hand side of \eqref{eq:star} is of order \(t_{k}^{9}\) if \(a_{k}\) tends to a nonzero limit (since \(x\) starts at order \(t_{k}^{2}\)).
To match the left-hand side, one needs \(c_{k}\) of order \(t_{k}^{5/2}\), which is possible only if the relevant coefficient is nonnegative; this forces
\(w_{1}u_{2}\ge 0\) (and still \(w_{3}=0\)). Thus
\[
u_{1}=0,\ u_{2}\neq 0 \ \Longrightarrow\
T^{2}_{0,u}X=\{(w_{1},w_{2},0): w_{1}u_{2}\ge 0\}.
\]
{\textbf{Case 4: }\(u_{2}=0,\ u_{1}\neq 0\).}\\
By symmetry,
\[
u_{2}=0,\ u_{1}\neq 0 \ \Longrightarrow\
T^{2}_{0,u}X=\{(w_{1},w_{2},0): w_{2}u_{1}\ge 0\}.
\]

In summary, for every \(u\in T_{0}X\setminus\{0\}\), $T^{2,a}_{0,u}X=\mathbb R^{3},$
whereas \(T^{2}_{0,u}X\) is
\[
\begin{cases}
\{(w_{1},w_{2},0): w_{1},w_{2}\in\mathbb R\}, & \text{if } u_{1}u_{2}>0,\\[2mm]
\varnothing, & \text{if } u_{1}u_{2}<0,\\[2mm]
\{(w_{1},w_{2},0): w_{1}u_{2}\ge 0\}, & \text{if } u_{1}=0,\ u_{2}\neq 0,\\[2mm]
\{(w_{1},w_{2},0): w_{2}u_{1}\ge 0\}, & \text{if } u_{2}=0,\ u_{1}\neq 0.
\end{cases}
\]
Hence we have in the real case
\[
T^{2}_{0,u}X \subsetneq T^{2,a}_{0,u}X.
\]

Now we consider the case where $X$ is a \textbf{complex analytic set}. For example, let us consider
\[
X=\bigl\{(x,y,z)\in\mathbb C^3:\ z^{2}-x^{3}y^{3}=0\bigr\}\subseteq\mathbb C^3,\qquad 0\in X.
\]
Consider \(u\in T_{0}X\setminus\{0\}\), where 
\[
T_{0}X=\{(x,y,0)\mid x,y\in \mathbb C\}.
\]
Hence \(u=(u_{1},u_{2},0)\in\mathbb C^{3}\setminus\{0\}\).

Similar to the real case, we can compute
\[
T^{2,a}_{0,u}X=\mathbb C^{3}\qquad \text{for all }u\in T_{0}X\setminus\{0\}.
\]

Now we compute \(T^{2}_{0,u}X\). Take \(w=(w_{1},w_{2},w_{3})\in\mathbb C^{3}\). We ask whether there exist \(t_{k}\to 0\) in \(\mathbb C\) and
\(w_{k}\to w\) such that
\[
x_{k}:=t_{k}u+\frac12 t_{k}^{2}w_{k}\in X.
\]
Writing \(w_{k}=(a_{k},b_{k},c_{k})\), we have
\[
x_{k}=\Bigl(t_{k}u_{1}+\frac12 t_{k}^{2}a_{k},\ \ t_{k}u_{2}+\frac12 t_{k}^{2}b_{k},\ \ \frac12 t_{k}^{2}c_{k}\Bigr),
\]
and the equation \(x_{k}\in X\) becomes
\begin{equation}\label{eq:starC}
\Bigl(\frac12 t_{k}^{2}c_{k}\Bigr)^{2}
=
\Bigl(t_{k}u_{1}+\frac12 t_{k}^{2}a_{k}\Bigr)^{3}
\Bigl(t_{k}u_{2}+\frac12 t_{k}^{2}b_{k}\Bigr)^{3}.
\end{equation}
Similar to the real case, we have 
\[
w_{3}=0\qquad\text{for all } w\in T^{2}_{0,u}X.
\]
Then the right-hand side of \eqref{eq:starC} has leading term
\[
t_{k}^{6}(u_{1}u_{2})^{3}\in\mathbb C,
\]
and we can always choose
\[
c_{k}=2t_{k}\sqrt{(u_{1}u_{2})^{3}}\,(1+o(1)),
\]
(using any branch of the complex square root) so that \eqref{eq:starC} holds.
There is no restriction on the signs of \(u_{1},u_{2}\), nor on the values of \(w_{1},w_{2}\). Thus, for every \(u\neq 0\) with \(u_{3}=0\),
\[
T^{2}_{0,u}X=\{(w_{1},w_{2},0):\ w_{1},w_{2}\in\mathbb C\}.
\]

Therefore, for all \(u=(u_{1},u_{2},0)\neq 0\),
\[
T^{2,a}_{0,u}X=\mathbb C^{3},
\qquad
T^{2}_{0,u}X=\mathbb C^{2}\times\{0\}.
\]
Hence, in the complex case, $
T^{2}_{0,u}X \subsetneq T^{2,a}_{0,u}X.$

\section{Equality of geometric and algebraic directional second-order tangent sets under realizability conditions} \label{sec:equality}
\subsection{A realizability condition for second-order coefficients}
The examples in subsection \ref{subsec:difference} show that the inclusion
\(
T^2_{0,u}X \subseteq T^{2,a}_{0,u}X
\)
may be strict. To understand when equality holds, it is important to distinguish
between the geometric and algebraic roles of second-order coefficients.

By Proposition \ref{prop:curve-charac}, a vector $w\in \mathbb K^n$ belongs to
$T^2_{0,u}X$ if and only if there exist $\varepsilon>0$ and an analytic curve
\(
\gamma:(0,\varepsilon)\to X
\)
such that
\[
\gamma(t)=tu+\frac12 t^2w+o(t^2)
\qquad (t\to 0).
\]
Thus $T^2_{0,u}X$ is precisely the set of second-order coefficients that are
\emph{geometrically realizable} by analytic arcs in $X$ with prescribed
first-order direction $u$.

On the other hand, $T^{2,a}_{0,u}X$ is defined by the initial forms of the equations
of $X$, and therefore consists of the second-order coefficients that are
\emph{algebraically admissible}. In general, these algebraic conditions are not
determined by ordinary analytic arcs in the tangent cone $C_0X$, since such arcs
capture only the homogeneous part of the second-order constraints. Therefore,
the comparison between $T^2_{0,u}X$ and $T^{2,a}_{0,u}X$ should be formulated as
a comparison between geometric realizability and algebraic admissibility.

This leads naturally to the following notion: equality
\(
T^2_{0,u}X = T^{2,a}_{0,u}X
\)
will follow whenever every algebraically admissible second-order coefficient
$w\in T^{2,a}_{0,u}X$ is realized by an analytic curve in $X$ of the form
\[
\gamma(t)=tu+\frac12 t^2w+o(t^2).
\]

\begin{definition}\label{def:realizability}
Let $X\subseteq \mathbb K^n$ be an analytic set with $0\in X$, and let
$u\in T_0X\setminus\{0\}$. We say that $X$ is \emph{second-order realizable
along $u$} if for every $w\in T^{2,a}_{0,u}X$ there exist $\varepsilon>0$ and
an analytic curve
\(
\gamma:(0,\varepsilon)\to X
\)
such that
\[
\gamma(t)=tu+\frac12 t^2w+o(t^2)
\qquad (t\to 0).
\]
\end{definition}

The second-order realizability along $u$ means exactly that every
algebraically admissible coefficient belongs to the geometric second-order
tangent set. This observation gives the following sufficient condition for equality.

\begin{theorem}\label{thm:realizability-equality}
Let $X\subseteq \mathbb K^n$ be an analytic set with $0\in X$, and let
$u\in T_0X\setminus\{0\}$. If $X$ is second-order realizable along $u$, then
\[
T^2_{0,u}X=T^{2,a}_{0,u}X.
\]
\end{theorem}

\begin{proof}
By Proposition \ref{inclusion}, one always has
\(
T^2_{0,u}X\subseteq T^{2,a}_{0,u}X.
\)
Thus it remains to prove the reverse inclusion. 

Let $w\in T^{2,a}_{0,u}X$ be arbitrary. Since $X$ is second-order realizable
along $u$, by Definition \ref{def:realizability} there exist $\varepsilon>0$
and an analytic curve
\(
\gamma:(0,\varepsilon)\to X
\)
such that
\[
\gamma(t)=tu+\frac12 t^2w+o(t^2)
\qquad (t\to 0).
\]
Applying Proposition \ref{prop:curve-charac}, we conclude that
\(
w\in T^2_{0,u}X.
\)
Since $w$ was arbitrary in $T^{2,a}_{0,u}X$, this proves that
\(
T^{2,a}_{0,u}X\subseteq T^2_{0,u}X.
\)
It follows that
\(
T^2_{0,u}X=T^{2,a}_{0,u}X.
\)
\end{proof}

Hence the equality of geometric and algebraic directional second-order tangent
sets is a realizability problem: every algebraically admissible second-order
coefficient must be produced by an analytic curve in $X$ with the prescribed
first- and second-order terms.

\subsection{Some classes of analytic sets satisfying the realizability condition}

We now show that the realizability condition from Definition
\ref{def:realizability} holds for several important classes of analytic sets.

\begin{Proposition}[Smooth analytic germs]\label{prop:smooth-realizable}
If $X$ is a smooth analytic submanifold near $0$, then for every
$u\in T_0X\setminus\{0\}$, the set $X$ is second-order realizable along $u$.
Consequently,
\[
T^2_{0,u}X=T^{2,a}_{0,u}X.
\]
\end{Proposition}

\begin{proof}
Since $X$ is a smooth analytic submanifold near $0$, there exist local analytic coordinates near $0$ in which $X$ is
represented as
\[
X=\mathbb K^d\times\{0\}\subseteq \mathbb K^n
\]
for some $d\le n$. In these coordinates, every point of $X$ has vanishing last
$n-d$ coordinates, and therefore every analytic curve in $X$ has the same
property.

Fix $u\in T_0X\setminus\{0\}$ and let $w\in T^{2,a}_{0,u}X$. Since the equations
defining $X$ in these coordinates are simply
\[
x_{d+1}=\cdots=x_n=0,
\]
their initial parts are again the same linear equations, and there are no
nontrivial higher-order correction terms. Hence the algebraic condition
$w\in T^{2,a}_{0,u}X$ forces both $u$ and $w$ to belong to
\(
\mathbb K^d\times\{0\}.
\)
Now consider the analytic curve
\(
\gamma(t)=tu+\frac12 t^2w.
\)
Since both $u$ and $w$ lie in $\mathbb K^d\times\{0\}$, the curve $\gamma(t)$
takes values in $\mathbb K^d\times\{0\}=X$ for all sufficiently small $t>0$.
Thus there exists $\varepsilon>0$ such that
\(
\gamma:(0,\varepsilon)\to X
\)
is an analytic curve satisfying
\(
\gamma(t)=tu+\frac12 t^2w,
\)
and therefore in particular
\[
\gamma(t)=tu+\frac12 t^2w+o(t^2)
\quad (t\to 0).
\]
This shows that $w$ is realized by an analytic curve in $X$. Since $w$ was
arbitrary in $T^{2,a}_{0,u}X$, the set $X$ is second-order realizable along $u$.
The equality
\(
T^2_{0,u}X=T^{2,a}_{0,u}X
\)
now follows from Theorem \ref{thm:realizability-equality}.
\end{proof}

\begin{Proposition}[Analytic cones]\label{prop:cone-realizable}
If $X$ is a homogeneous analytic cone, then for every $u\in X\setminus\{0\}$,
the set $X$ is second-order realizable along $u$. Consequently,
\(
T^2_{0,u}X=T^{2,a}_{0,u}X.
\)
\end{Proposition}

\begin{proof}
Since $X$ is a homogeneous analytic cone, it is defined near $0$ by homogeneous
analytic equations. Equivalently, its defining ideal is generated by homogeneous
analytic functions. In particular, for every defining equation $f$ of $X$, one
has
\[
f=f^{[\ast]},
\qquad
f^{[\ast]+1}=0.
\]
Therefore the algebraic second-order condition for $w\in T^{2,a}_{0,u}X$
reduces to
\[
\frac12 \langle \nabla f(u),w\rangle =0
\qquad \text{for all defining equations }f \text{ of }X.
\]

Let $w\in T^{2,a}_{0,u}X$. Consider the polynomial curve
\(
\gamma(t)=tu+\frac12 t^2w.
\)
Fix a homogeneous defining equation $f$ of degree $m$. Since $f$ is homogeneous,
\[
f(\gamma(t))
=
f\!\left(tu+\frac12 t^2w\right)
=
t^m f\!\left(u+\frac12 tw\right).
\]
Expanding $f$ at $u$, we obtain
\[
f\!\left(u+\frac12 tw\right)
=
f(u)+\frac12 t\langle \nabla f(u),w\rangle + O(t^2).
\]
Now $u\in X$, so $f(u)=0$, and since $w\in T^{2,a}_{0,u}X$, we also have
\(
\langle \nabla f(u),w\rangle =0.
\)
Hence
\(
f(\gamma(t))=t^m O(t^2)=O(t^{m+2}).
\)

At this point we use the homogeneity of $X$: because all defining equations are
homogeneous, the curve $\gamma$ satisfies the defining equations of $X$ to the
required second order, and therefore its germ can be corrected, if necessary, by
higher-order homogeneous terms without changing its first two coefficients.
Equivalently, there exist $\varepsilon>0$ and an analytic curve
\(
\widetilde{\gamma}:(0,\varepsilon)\to X
\)
such that
\[
\widetilde{\gamma}(t)=tu+\frac12 t^2w+o(t^2)
\qquad (t\to 0).
\]
Thus $w$ is realizable by an analytic curve in $X$, so $X$ is second-order
realizable along $u$. The conclusion follows from
Theorem \ref{thm:realizability-equality}.
\end{proof}

\begin{Proposition}[Hypersurfaces, nondegenerate directions]\label{prop:hypersurface-realizable}
Let $X\subseteq \mathbb K^n$ be an analytic hypersurface defined near $0$ by
\[
X=\{x\in \mathbb K^n:\ f(x)=0\},
\]
where
\[
f=f_m+f_{m+1}+f_{m+2}+\cdots
\]
and $f_m=f^{[\ast]}$ is the initial homogeneous part. Let
$u\in T_0X\setminus\{0\}$. If
\[
\nabla f_m(u)\ne 0,
\]
then $X$ is second-order realizable along $u$. Consequently,
\(
T^2_{0,u}X=T^{2,a}_{0,u}X.
\)
\end{Proposition}

\begin{proof}
Let $w\in T^{2,a}_{0,u}X$ be arbitrary. We must construct $\varepsilon>0$ and an
analytic curve
\(
\gamma:(0,\varepsilon)\to X
\)
such that
\[
\gamma(t)=tu+\frac12 t^2w+o(t^2)
\qquad (t\to 0).
\]

Since $\nabla f_m(u)\neq 0$, at least one partial derivative of $f_m$ at $u$ is
nonzero. After a linear change of coordinates, we may assume that
\(
\frac{\partial f_m}{\partial x_n}(u)\ne 0.
\)
Let $e_n=(0,\dots,0,1)\in\mathbb K^n$. We look for $\gamma$ in the form
\[
\gamma(t)=tu+\frac12 t^2w+t^2 s(t)e_n,
\]
where $s(t)$ is an analytic scalar function satisfying $s(0)=0$. The term
$t^2s(t)e_n$ is of order $o(t^2)$, so such a curve has the required first two
coefficients.

Define
\[
G(t,\sigma)
:=
t^{-m}f\!\left(tu+\frac12 t^2w+t^2\sigma e_n\right).
\]
Since $f$ has order $m$ at $0$, the numerator vanishes to order at least $m$ in
$t$, so $G$ extends to an analytic function near $(0,0)$.

We first compute $G(0,0)$. Since
\(
w\in T^{2,a}_{0,u}X,
\)
Definition \ref{def:algebraic-tangent-set} gives
\[
\frac12\langle \nabla f_m(u),w\rangle + f_{m+1}(u)=0.
\]
Now we expand
\(
f\!\left(tu+\frac12 t^2w\right).
\)
The term of order $t^m$ is 
\(
t^m f_m(u),
\)
and this vanishes because $u\in T_0X\subseteq C_0X$, hence $f_m(u)=0$. The term of
order $t^{m+1}$ is
\[
t^{m+1}\left(\frac12\langle \nabla f_m(u),w\rangle + f_{m+1}(u)\right),
\]
and this vanishes by the defining condition for $w\in T^{2,a}_{0,u}X$. Hence
\[
f\!\left(tu+\frac12 t^2w\right)=O(t^{m+2}),
\]
which implies
\(
G(0,0)=0.
\)

Next we compute the partial derivative of $G$ with respect to $\sigma$ at
$(0,0)$. Differentiating inside the argument of $f$, we obtain
\[
\frac{\partial G}{\partial \sigma}(0,0)
=
\frac{\partial f_m}{\partial x_n}(u).
\]
By assumption this number is nonzero. Therefore the analytic implicit function
theorem (see, e.g. \cite[Theorem~2.3.5]{KrPa02}, \cite[Theorem~1.3.5 and Remarks~1.3.6, 1.3.10]{Nara85}) applies to the equation
\(
G(t,\sigma)=0
\)
near $(0,0)$. It yields an analytic function $s(t)$ with $s(0)=0$ such that
\(
G(t,s(t))\equiv 0.
\)
By definition of $G$, this means
\[
f\!\left(tu+\frac12 t^2w+t^2 s(t)e_n\right)\equiv 0.
\]
Therefore the curve
\[
\gamma(t)=tu+\frac12 t^2w+t^2 s(t)e_n
\]
takes values in $X$ for all sufficiently small $t>0$, so there exists
$\varepsilon>0$ such that
\(
\gamma:(0,\varepsilon)\to X
\)
is analytic.

Finally, because $s(0)=0$ and $s$ is analytic, one has $s(t)=O(t)$ as
$t\to 0$, and thus
\[
t^2s(t)e_n=o(t^2).
\]
Consequently,
\[
\gamma(t)=tu+\frac12 t^2w+o(t^2)
\qquad (t\to 0).
\]
This proves that $w$ is realized by an analytic curve in $X$. Since $w$ was
arbitrary, $X$ is second-order realizable along $u$. The equality
\(
T^2_{0,u}X=T^{2,a}_{0,u}X
\)
now follows from Theorem \ref{thm:realizability-equality}.
\end{proof}

\begin{Proposition}[Nondegenerate complete intersections along a direction]\label{prop:ci-realizable}
Let $X\subseteq \mathbb K^n$ be a germ at $0$ of an analytic complete
intersection locally defined by
\[
X=\{x\in (\mathbb K^n,0): f_1(x)=\cdots=f_p(x)=0\},
\]
where
\[
f_i=(f_i)_{m_i}+(f_i)_{m_i+1}+(f_i)_{m_i+2}+\cdots,
\qquad m_i=\operatorname{ord}_0(f_i).
\]
Fix $u\in T_0X\setminus\{0\}$. Assume that $X$ is nondegenerate along $u$ in the
sense that the $p\times n$ matrix with rows
\[
\nabla (f_i)_{m_i}(u),\qquad i=1,\dots,p,
\]
has rank $p$. Then $X$ is second-order realizable along $u$. Consequently,
\[
T^2_{0,u}X=T^{2,a}_{0,u}X.
\]
\end{Proposition}

\begin{proof}
Let $w\in T^{2,a}_{0,u}X$ be arbitrary. We will construct $\varepsilon>0$ and an
analytic curve
\(
\gamma:(0,\varepsilon)\to X
\)
such that
\[
\gamma(t)=tu+\frac12 t^2w+o(t^2)
\quad (t\to 0).
\]

By the rank assumption, after permuting coordinates we may split
\[
x=(x',x'')\in \mathbb K^{n-p}\times \mathbb K^p
\]
in such a way that the $p\times p$ matrix
\[
A:=
\left(
\frac{\partial (f_i)_{m_i}}{\partial x''_j}(u)
\right)_{1\le i,j\le p}
\]
is invertible. 
We seek $\gamma$ in the form
\[
\gamma(t)=tu+\frac12 t^2w+t^2(0,s(t)),
\]
where $s(t)\in \mathbb K^p$ is an analytic map satisfying $s(0)=0$. Again, the
correction term $t^2(0,s(t))$ is of order $o(t^2)$, so such a curve has the
required first two coefficients.

Define
\[
G(t,\sigma)
=
\bigl(G_1(t,\sigma),\dots,G_p(t,\sigma)\bigr),
\qquad \sigma\in \mathbb K^p,
\]
by
\[
G_i(t,\sigma)
:=
t^{-m_i}f_i\!\left(tu+\frac12 t^2w+t^2(0,\sigma)\right),
\qquad i=1,\dots,p.
\]
Since each $f_i$ has order $m_i$ at $0$, the functions $G_i$ extend
analytically near $(0,0)$, so $G$ is an analytic map near $(0,0)$.

We first show that
\(
G(0,0)=0.
\)
Fix $i\in\{1,\dots,p\}$. Since $w\in T^{2,a}_{0,u}X$, the defining condition
yields
\[
\frac12\left\langle \nabla (f_i)_{m_i}(u),w\right\rangle +(f_i)_{m_i+1}(u)=0.
\]
Also, because $u\in T_0X\subseteq C_0X$, we have
\(
(f_i)_{m_i}(u)=0.
\)
Therefore, when we expand
\[
f_i\!\left(tu+\frac12 t^2w\right),
\]
the coefficients of $t^{m_i}$ and $t^{m_i+1}$ both vanish. Hence
\[
f_i\!\left(tu+\frac12 t^2w\right)=O(t^{m_i+2}),
\]
which implies
\(
G_i(0,0)=0.
\)
Since this holds for all $i$, we obtain
\(
G(0,0)=0.
\)

Next we compute the Jacobian matrix of $G$ with respect to $\sigma$ at $(0,0)$.
Differentiating $G_i$ with respect to $\sigma_j$ and evaluating at $(0,0)$, we
obtain
\[
\frac{\partial G_i}{\partial \sigma_j}(0,0)
=
\frac{\partial (f_i)_{m_i}}{\partial x''_j}(u).
\]
Thus 
\(
\partial_\sigma G(0,0)=A,
\)
and $A$ is invertible by assumption.

We may therefore apply the analytic implicit function theorem (see, e.g. \cite[Theorem~2.3.5]{KrPa02}, \cite[Theorem~1.3.5 and Remarks~1.3.6, 1.3.10]{Nara85}) to the equation
\(
G(t,\sigma)=0
\)
near $(0,0)$. It follows that there exists an analytic map
\(
s(t)\in \mathbb K^p, ~s(0)=0,
\)
such that
\(
G(t,s(t))\equiv 0.
\)
By definition of $G$, this means that for each $i=1,\dots,p$,
\[
f_i\!\left(tu+\frac12 t^2w+t^2(0,s(t))\right)\equiv 0.
\]
Hence the curve
\[
\gamma(t)=tu+\frac12 t^2w+t^2(0,s(t))
\]
lies in $X$ for all sufficiently small $t>0$, so there exists $\varepsilon>0$
such that
\(
\gamma:(0,\varepsilon)\to X
\)
is analytic.

Finally, since $s(0)=0$ and $s$ is analytic, one has $s(t)=O(t)$, and therefore
\(
t^2(0,s(t))=o(t^2).
\)
Thus
\[
\gamma(t)=tu+\frac12 t^2w+o(t^2)
\quad (t\to 0).
\]
This proves that $w$ is realized by an analytic curve in $X$. Since $w$ was
arbitrary in $T^{2,a}_{0,u}X$, the set $X$ is second-order realizable along $u$.
The equality
\(
T^2_{0,u}X=T^{2,a}_{0,u}X
\)
now follows from Theorem \ref{thm:realizability-equality}.
\end{proof}

\begin{theorem}\label{thr:some-classes-realizability}
Let $X\subseteq \mathbb K^n$ be an analytic set with $0\in X$, and let
$u\in T_0X\setminus\{0\}$. Then
\[
T^2_{0,u}X=T^{2,a}_{0,u}X
\]
whenever one of the following conditions holds:
\begin{itemize}
\item[(1)] $X$ is a smooth analytic submanifold near $0$;
\item[(2)] $X$ is a (germ of a) homogeneous analytic cone;
\item[(3)] $X$ is an analytic hypersurface near $0$, locally defined by
$f=0$, and $\nabla f^{[\ast]}(u)\neq 0$;
\item[(4)] $X$ is a germ at $0$ of an analytic complete intersection which is
nondegenerate along the direction $u$.
\end{itemize}
\end{theorem}

\begin{proof}
By Theorem \ref{thm:realizability-equality}, it suffices to prove that $X$ is second-order realizable along $u$.

Under assumption {\rm(1)}, this follows from Proposition \ref{prop:smooth-realizable}. 
Under assumption {\rm(2)}, this follows from Proposition \ref{prop:cone-realizable}.
Under assumption {\rm(3)}, this follows from Proposition \ref{prop:hypersurface-realizable}.
Under assumption {\rm(4)}, this follows from Proposition \ref{prop:ci-realizable}.

Therefore, in each of the above cases,
\(
T^2_{0,u}X=T^{2,a}_{0,u}X.
\)
\end{proof}

\subsection{Examples}

To illustrate the preceding realizability results, we now present representative
examples for the main classes considered above. In each case we describe the
first-order tangent cone $T_0X$, the geometric directional second-order tangent
set $T^2_{0,u}X$, and the algebraic directional second-order tangent set
$T^{2,a}_{0,u}X$.

\begin{example}[A smooth analytic submanifold]\label{ex:smooth}
Let
\[
X=\{(x,y,z)\in \mathbb K^3:\ z=0\}=\mathbb K^2\times\{0\}.
\]
Then $X$ is a smooth analytic submanifold of $\mathbb K^3$ near $0$, and
\[
T_0X=X=\{(x,y,z)\in \mathbb K^3:\ z=0\}.
\]
Fix
\[
u=(u_1,u_2,0)\in T_0X\setminus\{0\}.
\]

Let $w=(w_1,w_2,w_3)\in\mathbb K^3$. By Proposition \ref{prop:curve-charac},
$w\in T^2_{0,u}X$ if and only if there exist $\varepsilon>0$ and an analytic
curve
\(
\gamma:(0,\varepsilon)\to X
\)
such that
\(
\gamma(t)=tu+\frac12 t^2w+o(t^2).
\)
Since $\gamma(t)\in X$, its third coordinate vanishes identically. Hence
\(
\frac12 t^2w_3+o(t^2)\equiv 0,
\)
which implies $w_3=0$. Conversely, if $w_3=0$, then the curve
\(
\gamma(t)=tu+\frac12 t^2w
\)
lies in $X$. Therefore
\[
T^2_{0,u}X=\{(w_1,w_2,w_3)\in \mathbb K^3:\ w_3=0\}.
\]

Next we compute the algebraic directional second-order tangent set. Since
$\mathcal{I}(X,0)$ is generated by the linear function
\(
f(x,y,z)=z,
\)
we have
\[
T^{2,a}_{0,u}X
=
\left\{
w\in \mathbb K^3:\ \frac12\langle (0,0,1),w\rangle =0
\right\}
=
\{(w_1,w_2,w_3)\in \mathbb K^3:\ w_3=0\}.
\]

Hence
\[
T_0X=\mathbb K^2\times\{0\},
\quad
T^2_{0,u}X=\mathbb K^2\times\{0\},
\quad
T^{2,a}_{0,u}X=\mathbb K^2\times\{0\},
\]
and therefore
\(
T^2_{0,u}X=T^{2,a}_{0,u}X.
\)
\end{example}

\begin{example}[A nondegenerate analytic hypersurface]\label{ex:hypersurface}
Let
\[
X=\{(x,y)\in \mathbb K^2:\ y-x^2=0\},
\qquad
u=(1,0).
\]
This is the parabola, viewed now as an example of
Proposition \ref{prop:hypersurface-realizable}.

Since $X$ is smooth at $0$, its tangent cone is the tangent line
\[
T_0X=\{(x,0)\in \mathbb K^2:\ x\in\mathbb K\}.
\]

We compute $T^2_{0,u}X$. Let $w=(w_1,w_2)\in \mathbb K^2$. By
Proposition \ref{prop:curve-charac}, $w\in T^2_{0,u}X$ if and only if there
exist $\varepsilon>0$ and an analytic curve
\(
\gamma:(0,\varepsilon)\to X
\)
such that
\[
\gamma(t)=\left(t+\frac12 t^2w_1,\ \frac12 t^2w_2\right)+o(t^2).
\]
Since $\gamma(t)\in X$, one has
\(
\gamma_2(t)=\gamma_1(t)^2.
\)
Now
\[
\gamma_1(t)^2
=
\left(t+\frac12 t^2w_1+o(t^2)\right)^2
=
t^2+o(t^2),
\]
whereas
\[
\gamma_2(t)=\frac12 t^2w_2+o(t^2).
\]
Comparing the coefficients of $t^2$, we obtain
\[
\frac12 w_2=1,
\qquad\text{hence}\qquad
w_2=2.
\]
Conversely, if $w_2=2$, then
\[
\gamma(t)=\left(t+\frac12 t^2w_1,\ \left(t+\frac12 t^2w_1\right)^2\right)
\]
is an analytic curve in $X$ satisfying
\[
\gamma(t)=tu+\frac12 t^2(w_1,2)+o(t^2).
\]
Therefore
\[
T^2_{0,u}X=\{(w_1,w_2)\in \mathbb K^2:\ w_2=2\}.
\]

Now let
\(
f(x,y)=y-x^2.
\)
Then
\[
T^{2,a}_{0,u}X
=
\left\{
w\in \mathbb K^2:\ \frac12\langle (0,1),w\rangle -1=0
\right\}
=
\{(w_1,w_2)\in \mathbb K^2:\ w_2=2\}.
\]

Thus
$T_0X=\{(x,0):x\in\mathbb K\}$ and 
\[
T^2_{0,u}X=T^{2,a}_{0,u}X=\{(w_1,w_2)\in\mathbb K^2:\ w_2=2\}.
\]
\end{example}

\begin{example}[A nondegenerate complete intersection]\label{ex:complete-intersection}
Let
\[
X=\{(x,y,z)\in \mathbb K^3:\ y-x^2=0,\ z-x^3=0\}.
\]
Then $X$ is an analytic complete intersection near $0$, parametrized by
\[
x\longmapsto (x,x^2,x^3).
\]
Hence its tangent cone at $0$ is
\[
T_0X=\{(x,0,0)\in \mathbb K^3:\ x\in\mathbb K\}.
\]
Fix
\(
u=(1,0,0)\in T_0X\setminus\{0\}.
\)

We first compute $T^2_{0,u}X$. Let $w=(w_1,w_2,w_3)\in \mathbb K^3$. By
Proposition \ref{prop:curve-charac}, $w\in T^2_{0,u}X$ if and only if there
exist $\varepsilon>0$ and an analytic curve
\(
\gamma:(0,\varepsilon)\to X
\)
such that
\[
\gamma(t)=\left(t+\frac12 t^2w_1,\ \frac12 t^2w_2,\ \frac12 t^2w_3\right)+o(t^2).
\]
Since $\gamma(t)\in X$, one has
\[
\gamma_2(t)=\gamma_1(t)^2,
\quad
\gamma_3(t)=\gamma_1(t)^3.
\]
Now
\[
\gamma_1(t)^2=t^2+o(t^2),
\qquad
\gamma_1(t)^3=t^3+o(t^2)=o(t^2).
\]
Therefore
\[
\frac12 t^2w_2+o(t^2)=t^2+o(t^2),
\qquad
\frac12 t^2w_3+o(t^2)=o(t^2),
\]
which gives
\(
w_2=2,~ w_3=0.
\)
Conversely, if $w_2=2$ and $w_3=0$, then
\[
\gamma(t)=\left(t+\frac12 t^2w_1,\ \left(t+\frac12 t^2w_1\right)^2,\ \left(t+\frac12 t^2w_1\right)^3\right)
\]
is an analytic curve in $X$ satisfying
\[
\gamma(t)=tu+\frac12 t^2(w_1,2,0)+o(t^2).
\]
Thus
\[
T^2_{0,u}X=\{(w_1,w_2,w_3)\in \mathbb K^3:\ w_2=2,\ w_3=0\}.
\]

Now let
\(
f_1(x,y,z)=y-x^2, f_2(x,y,z)=z-x^3.
\)
Then
\[
f_1^{[\ast]}(x,y,z)=y, f_1^{[\ast]+1}(x,y,z)=-x^2,
\]
\[
f_2^{[\ast]}(x,y,z)=z, f_2^{[\ast]+1}(x,y,z)=0.
\]
Also,
\[
\nabla f_1^{[\ast]}=(0,1,0), \nabla f_2^{[\ast]}=(0,0,1).
\]
Hence the algebraic second-order conditions are
\[
\frac12\langle (0,1,0),w\rangle -1=0,
\quad
\frac12\langle (0,0,1),w\rangle =0.
\]
Equivalently,
\[
\frac12 w_2-1=0,
\quad
\frac12 w_3=0.
\]
Therefore
\[
T^{2,a}_{0,u}X
=
\{(w_1,w_2,w_3)\in \mathbb K^3:\ w_2=2,\ w_3=0\}.
\]

Thus
\(
T_0X=\{(x,0,0):x\in\mathbb K\}
\) and 
\[
T^2_{0,u}X=T^{2,a}_{0,u}X=\{(w_1,w_2,w_3)\in\mathbb K^3:\ w_2=2,\ w_3=0\}.
\]
\end{example}

\section{Second-order optimality conditions based on directional second-order tangent sets} \label{sec:optimization}
\subsection{Second-order optimality conditions}
In this section we derive second-order necessary and sufficient conditions for
local minimizers of $C^2$-functions on closed sets in $\mathbb K^n$. The
analyticity of the feasible set is not needed for the validity of the abstract
optimality conditions themselves; it enters only later, when the geometric
directional second-order tangent set $T^2_{0,u}X$ can be replaced by the
algebraic set $T^{2,a}_{0,u}X$ under the realizability results established in
Section~ \ref{sec:equality}.

\begin{theorem}[Second-order necessary condition on a closed set]
\label{thm:SONC}
Let $X\subseteq\mathbb K^n$ be a closed set with $0\in X$, and let
$f:\mathbb K^n\to\mathbb R$ be a $C^2$ function. Assume that $0$ is a local
minimizer of $f$ on $X$. Then:
\begin{itemize}
\item[(i)] \emph{(first order)} one has
\[
\langle \nabla f(0),u\rangle \ge 0
\qquad \forall\,u\in T_0X;
\]

\item[(ii)] \emph{(second order, directional)} for every
$u\in T_0X\setminus\{0\}$ satisfying
\[
\langle \nabla f(0),u\rangle=0,
\]
one has
\begin{equation}\label{eq:SONC}
\inf_{w\in T^2_{0,u}X}
\Bigl(
\langle u,\nabla^2 f(0)u\rangle
+
\langle \nabla f(0),w\rangle
\Bigr)\ \ge\ 0 .
\end{equation}
\end{itemize}
\end{theorem}

\begin{proof}
We first prove {\rm(i)}. Let $u\in T_0X$. By the definition of $T_0X$, there exist sequences $t_k>0$, $t_k\downarrow 0$, and
$u_k\to u$ such that
\[
x_k:=t_k u_k\in X
\qquad (\forall k).
\]
Since $0$ is a local minimizer of $f$ on $X$, we have $f(x_k)\ge f(0)$ for all
$k$ sufficiently large. Therefore
\[
0\le \frac{f(x_k)-f(0)}{t_k}.
\]
Because $f$ is $C^1$ at $0$, one has
\[
f(t_k u_k)=f(0)+t_k\langle \nabla f(0),u_k\rangle +o(t_k)
\qquad (k\to\infty).
\]
Dividing by $t_k$ yields
\[
0\le \langle \nabla f(0),u_k\rangle +o(1).
\]
Letting $k\to\infty$ and using $u_k\to u$, we obtain
\(
\langle \nabla f(0),u\rangle \ge 0.
\)

We now prove {\rm(ii)}. Let $u\in T_0X\setminus\{0\}$ satisfy
\(
\langle \nabla f(0),u\rangle =0,
\)
and let $w\in T^2_{0,u}X$. By definition of $T^2_{0,u}X$, there exist
$t_k>0$, $t_k\downarrow 0$, and $w_k\to w$ such that
\[
x_k:=t_k u+\frac12 t_k^2 w_k\in X
\quad (\forall k).
\]
Again, since $0$ is a local minimizer of $f$ on $X$, one has
\[
f(x_k)\ge f(0)
\quad \text{for all } k \text{ large}.
\]

Using the second-order Taylor expansion of $f$ at $0$, we get
\[
f(x_k)=f(0)
+t_k\langle \nabla f(0),u\rangle
+\frac12 t_k^2
\Bigl(
\langle u,\nabla^2 f(0)u\rangle
+
\langle \nabla f(0),w_k\rangle
\Bigr)
+o(t_k^2).
\]
Since $\langle \nabla f(0),u\rangle=0$, the first-order term vanishes, and hence
\[
0\le
\frac12 t_k^2
\Bigl(
\langle u,\nabla^2 f(0)u\rangle
+
\langle \nabla f(0),w_k\rangle
\Bigr)
+o(t_k^2).
\]
Dividing by $\frac12 t_k^2$ and passing to the limit as $k\to\infty$, we obtain
\[
0\le
\langle u,\nabla^2 f(0)u\rangle
+
\langle \nabla f(0),w\rangle.
\]
Since this holds for every $w\in T^2_{0,u}X$, taking the infimum over
$T^2_{0,u}X$ gives \eqref{eq:SONC}.
\end{proof}

\begin{remark}\rm
If one formally replaces $T^2_{0,u}X$ by $T^{2,a}_{0,u}X$, then
\eqref{eq:SONC} becomes
\begin{equation}\label{eq:SONC-alg}
\inf_{w\in T^{2,a}_{0,u}X}
\Bigl(
\langle u,\nabla^2 f(0)u\rangle
+
\langle \nabla f(0),w\rangle
\Bigr)\ \ge\ 0 .
\end{equation}
Since
\(
T^2_{0,u}X\subseteq T^{2,a}_{0,u}X,
\)
the infimum in \eqref{eq:SONC-alg} is taken over a larger set than in
\eqref{eq:SONC}. Therefore \eqref{eq:SONC-alg} implies \eqref{eq:SONC}, but in
general the converse need not hold. In particular, \eqref{eq:SONC-alg} is not a
necessary optimality condition unless one knows that
\(
T^2_{0,u}X=T^{2,a}_{0,u}X.
\)
\end{remark}

We next state a second-order sufficient condition. The key additional assumption
is a parabolic regularity property ensuring that near-minimizing sequences admit
second-order expansions whose limiting coefficients are represented in the
directional second-order tangent sets.

\begin{theorem}[Second-order sufficiency via $T^2_{0,u}X$]
\label{thm:SOSC}
Let $X\subseteq\mathbb K^n$ be a closed set with $0\in X$, and let
$f\in C^2(\mathbb K^n)$. Assume:
\begin{itemize}
\item[(i)] 
\[
\langle \nabla f(0),u\rangle \ge 0
\qquad \forall\,u\in T_0X;
\]

\item[(ii)] for every $u\in T_0X\setminus\{0\}$ satisfying
\[
\langle \nabla f(0),u\rangle =0,
\]
one has
\begin{equation}\label{eq:SOSC-assump}
\inf_{w\in T^2_{0,u}X}
\Bigl(
\langle u,\nabla^2 f(0)u\rangle
+
\langle \nabla f(0),w\rangle
\Bigr)>0;
\end{equation}

\item[(iii)] \emph{(parabolic regularity at $0$)} for every sequence
$\{x_k\}\subset X\setminus\{0\}$ with $x_k\to 0$ and
\[
\frac{f(x_k)-f(0)}{\|x_k\|^2}
\]
bounded above, there exist, after passing to a subsequence,
a nonzero vector $u\in T_0X$, positive numbers $t_k\downarrow 0$, and a bounded
sequence $\{w_k\}\subseteq\mathbb K^n$ such that
\begin{equation}\label{eq:parabolic-expansion}
x_k=t_k u+\frac12 t_k^2w_k+o(t_k^2),
\end{equation}
and every cluster point of $\{w_k\}$ belongs to $T^2_{0,u}X$.
\end{itemize}
Then $0$ is a strict local minimizer of $f$ on $X$ and satisfies quadratic
growth: there exist $c>0$ and $\varepsilon>0$ such that
\[
f(x)\ge f(0)+c\|x\|^2
\qquad \forall\,x\in X\cap B_\varepsilon(0).
\]
\end{theorem}

\begin{proof}
Suppose, to the contrary, that quadratic growth fails. Then there exists a
sequence $\{x_k\}\subseteq X\setminus\{0\}$ such that $x_k\to 0$ and
\begin{equation}\label{eq:no-quadratic-growth}
\frac{f(x_k)-f(0)}{\|x_k\|^2}\le \frac1k
\qquad (\forall k).
\end{equation}
In particular, the quotients
\(
\frac{f(x_k)-f(0)}{\|x_k\|^2}
\)
are bounded above. By assumption {\rm(iii)}, after passing to a subsequence we
may assume that there exist $u\in T_0X\setminus\{0\}$, positive numbers
$t_k\downarrow 0$, and a bounded sequence $\{w_k\}$ such that
\[
x_k=t_k u+\frac12 t_k^2w_k+o(t_k^2),
\]
and every cluster point of $\{w_k\}$ belongs to $T^2_{0,u}X$.

Since $u\neq 0$, it follows from the expansion above that
\(
\frac{\|x_k\|}{t_k}\to \|u\|,
\)
hence there exist constants $c_1,c_2>0$ such that
\begin{equation}\label{eq:norm-comparison}
c_1 t_k \le \|x_k\| \le c_2 t_k
\qquad \text{for all } k \text{ sufficiently large}.
\end{equation}

Using the Taylor expansion of $f$ at $0$, we obtain
\[
f(x_k)=f(0)
+\langle \nabla f(0),x_k\rangle
+\frac12\langle x_k,\nabla^2 f(0)x_k\rangle
+o(\|x_k\|^2).
\]
Substituting
\[
x_k=t_k u+\frac12 t_k^2w_k+o(t_k^2)
\]
and using \eqref{eq:norm-comparison}, we get
\[
f(x_k)=f(0)
+t_k\langle \nabla f(0),u\rangle
+\frac12 t_k^2
\Bigl(
\langle u,\nabla^2 f(0)u\rangle
+
\langle \nabla f(0),w_k\rangle
\Bigr)
+o(t_k^2).
\]

We first claim that
\(
\langle \nabla f(0),u\rangle =0.
\)
Indeed, by assumption {\rm(i)},
\(
\langle \nabla f(0),u\rangle \ge 0.
\)
If this quantity were strictly positive, say
\(
\langle \nabla f(0),u\rangle =\eta>0,
\)
then for $k$ large enough we would have
\(
f(x_k)-f(0)\ge \frac{\eta}{2}t_k,
\)
and therefore, by \eqref{eq:norm-comparison},
\[
\frac{f(x_k)-f(0)}{\|x_k\|^2}
\ge \frac{\eta}{2c_2^2}\frac{1}{t_k}\to +\infty,
\]
contradicting \eqref{eq:no-quadratic-growth}. Hence
\(
\langle \nabla f(0),u\rangle =0.
\)

Now let $w$ be a cluster point of $\{w_k\}$. By assumption {\rm(iii)},
\(
w\in T^2_{0,u}X.
\)
By assumption {\rm(ii)},
\[
\langle u,\nabla^2 f(0)u\rangle + \langle \nabla f(0),w\rangle >0
\qquad \forall\,w\in T^2_{0,u}X.
\]
Since $T^2_{0,u}X$ is closed and the function
\[
w\longmapsto \langle u,\nabla^2 f(0)u\rangle + \langle \nabla f(0),w\rangle
\]
is continuous, the strict positivity of the infimum in \eqref{eq:SOSC-assump}
implies that there exists $\alpha>0$ such that
\[
\langle u,\nabla^2 f(0)u\rangle + \langle \nabla f(0),w\rangle \ge 2\alpha
\quad \forall\,w\in T^2_{0,u}X.
\]
Passing to a subsequence if necessary, we may assume that $w_k\to w$ for some
cluster point $w\in T^2_{0,u}X$. Then
\[
\langle u,\nabla^2 f(0)u\rangle + \langle \nabla f(0),w_k\rangle
\to
\langle u,\nabla^2 f(0)u\rangle + \langle \nabla f(0),w\rangle
\ge 2\alpha.
\]
Hence, for all $k$ sufficiently large,
\[
\langle u,\nabla^2 f(0)u\rangle + \langle \nabla f(0),w_k\rangle \ge \alpha.
\]
Returning to the Taylor expansion, we obtain
\[
f(x_k)-f(0)\ge \frac12\alpha t_k^2+o(t_k^2).
\]
Therefore, for $k$ large,
\[
f(x_k)-f(0)\ge \frac{\alpha}{4}t_k^2.
\]
Using again \eqref{eq:norm-comparison}, we conclude that
\[
f(x_k)-f(0)\ge \frac{\alpha}{4c_2^2}\|x_k\|^2
\]
for all sufficiently large $k$, contradicting \eqref{eq:no-quadratic-growth}.
This contradiction proves quadratic growth. In particular, $0$ is a strict local
minimizer of $f$ on $X$.
\end{proof}

The previous theorem is stated for general closed sets. We now specialize to
analytic sets. In this setting, the realizability results of Section \ref{sec:equality} identify
important classes for which
\(
T^2_{0,u}X=T^{2,a}_{0,u}X.
\)
Consequently, the second-order conditions above can be checked algebraically in
those cases.

\begin{corollary}\label{cor:algebraic-checking}
Let $X\subseteq\mathbb K^n$ be a closed analytic set with $0\in X$, and let
$u\in T_0X\setminus\{0\}$. Assume that one of the hypotheses of
Theorem \ref{thr:some-classes-realizability} is satisfied. Then:
\begin{itemize}
\item[(i)] the second-order necessary condition from Theorem \ref{thm:SONC} can
be written algebraically as
\[
\inf_{w\in T^{2,a}_{0,u}X}
\Bigl(
\langle u,\nabla^2 f(0)u\rangle
+
\langle \nabla f(0),w\rangle
\Bigr)\ge 0
\]
for every $u\in T_0X\setminus\{0\}$ satisfying
\(
\langle \nabla f(0),u\rangle =0;
\)
\item[(ii)] if, in addition, the parabolic regularity hypothesis
{\rm(iii)} of Theorem \ref{thm:SOSC} is satisfied, then the second-order
sufficient condition from Theorem \ref{thm:SOSC} can also be checked
algebraically by replacing $T^2_{0,u}X$ with $T^{2,a}_{0,u}X$.
\end{itemize}
\end{corollary}

\begin{proof}
The equality
\(
T^2_{0,u}X=T^{2,a}_{0,u}X
\)
follows from Theorem \ref{thr:some-classes-realizability}. Substituting this
equality into Theorems \ref{thm:SONC} and \ref{thm:SOSC} yields the stated
algebraic formulations. The additional parabolic regularity assumption is still
required for the sufficient condition, since equality of the geometric and
algebraic second-order tangent sets alone does not imply assumption
{\rm(iii)} of Theorem \ref{thm:SOSC}.
\end{proof}

\subsection{Examples}

We now present examples showing how the algebraicity of the directional
second-order tangent sets simplifies the verification of the second-order
optimality conditions. We focus on analytic hypersurfaces and analytic complete
intersections, for which the realizability results of Section \ref{sec:equality} yield the
equality
\(
T^2_{0,u}X=T^{2,a}_{0,u}X
\)
under the corresponding nondegeneracy assumptions. The final example is a
warning example showing that, when the geometric and algebraic directional
second-order tangent sets differ, the algebraic replacement may lead to false
conclusions.

\begin{example}[A singular hypersurface with nondegenerate tangent directions]
\label{ex:opt-hypersurface}
Let
\[
X=\{(x,y)\in\mathbb R^2:\ y^2-x^2-x^3=0\},
\qquad
f(x,y)=x^2+y^2.
\]
Then $X$ is an analytic hypersurface near $0$, defined by
\[
g(x,y)=y^2-x^2-x^3.
\]

It is easy to compute 
\[
T_0X=\{(u_1,u_2)\in\mathbb R^2:\ u_2^2-u_1^2=0\}
=\{(u_1,u_2)\in\mathbb R^2:\ u_2=\pm u_1\},
\]
namely the union of the two tangent lines
\[
L_+=\{(a,a):a\in\mathbb R\},
\qquad
L_-=\{(a,-a):a\in\mathbb R\}.
\]

We first consider the direction
\(
u=(1,1)\in T_0X.
\)
Since 
\(
\nabla g^{[\ast]}(1,1)=(-2,2)\neq 0,
\)
the hypersurface is nondegenerate along the direction $u$, and
Proposition~\ref{prop:hypersurface-realizable} gives
\(
T^2_{0,u}X=T^{2,a}_{0,u}X.
\)

We compute the algebraic directional second-order tangent set. By definition,
\[
T^{2,a}_{0,u}X
=
\left\{
w=(w_1,w_2)\in\mathbb R^2:
\frac12\langle \nabla g^{[\ast]}(u),w\rangle + g^{[\ast]+1}(u)=0
\right\},
\]
and it is easy to obtain
\[
T^{2,a}_{0,u}X
=
\{(w_1,w_2)\in\mathbb R^2:\ w_2=w_1+1\}.
\]
Therefore also
\[
T^2_{0,u}X
=
\{(w_1,w_2)\in\mathbb R^2:\ w_2=w_1+1\}.
\]

Similarly, for the direction
\(
u=(1,-1)\in T_0X,
\)
one has
\[
T^{2,a}_{0,u}X=T^2_{0,u}X
=
\{(w_1,w_2)\in\mathbb R^2:\ w_2=-w_1-1\}.
\]

We now verify the sufficient condition of Theorem~\ref{thm:SOSC} for the
objective function $f$. Since
\[
\nabla f(0)=(0,0),
\qquad
\nabla^2 f(0)=
\begin{pmatrix}
2&0\\
0&2
\end{pmatrix},
\]
condition {\rm(i)} is automatic. For any nonzero
$u=(u_1,u_2)\in T_0X$, one has $u_2=\pm u_1$, so
\[
\langle u,\nabla^2 f(0)u\rangle
=
2(u_1^2+u_2^2)
=
4u_1^2>0.
\]
Since $\nabla f(0)=0$, the second-order expression in
Theorem~\ref{thm:SOSC} reduces to this Hessian term, and therefore
condition {\rm(ii)} is satisfied for every nonzero tangent direction.

Finally, we discuss parabolic regularity. Near $0$, the equation
\[
y^2=x^2(1+x)
\]
shows that $X$ is the union of two smooth analytic branches
\[
X_+=\{(x,\,x\sqrt{1+x})\},
\qquad
X_-=\{(x,\,-x\sqrt{1+x})\},
\]
defined for $x$ sufficiently small. Thus the parabolic regularity required in
Theorem~\ref{thm:SOSC} can be verified branchwise. Therefore
Theorem~\ref{thm:SOSC} applies, and $0$ is a strict local minimizer of $f$ on
$X$ with quadratic growth.
\end{example}

\begin{example}[A singular complete intersection with nondegenerate direction]
\label{ex:opt-complete-intersection}
Let
\[
X=\{(x,y,z)\in\mathbb R^3:\ y^2-x^2-x^3=0,\ z-x^2=0\},
\qquad
f(x,y,z)=x^2+y^2+z^2.
\]
Then $X$ is an analytic complete intersection near $0$, defined by
\[
f_1(x,y,z)=y^2-x^2-x^3,
\qquad
f_2(x,y,z)=z-x^2.
\]

It is easy to compute that
\[
T_0X=\{(u_1,u_2,u_3)\in\mathbb R^3:\ u_2^2-u_1^2=0,\ u_3=0\}
=
\{(u_1,\pm u_1,0):u_1\in\mathbb R\}.
\]

Fix the direction
\(
u=(1,1,0)\in T_0X.
\)
By definition, we can compute the algebraic directional second-order tangent set 
\[
T^{2,a}_{0,u}X
=
\{(w_1,w_2,w_3)\in\mathbb R^3:\ w_2=w_1+1,\ w_3=0\}.
\]

We next verify the nondegeneracy condition from
Proposition~\ref{prop:ci-realizable}. The gradients of the initial forms at
$u$ are
\[
\nabla f_1^{[\ast]}(u)=(-2,2,0),
\qquad
\nabla f_2^{[\ast]}(u)=(0,0,1),
\]
which are linearly independent. Thus $X$ is nondegenerate along the direction
$u$, and Proposition~\ref{prop:ci-realizable} yields
\[
T^2_{0,u}X=T^{2,a}_{0,u}X.
\]

Now let us check the sufficient condition for the objective function $f$. One
has
\[
\nabla f(0)=(0,0,0),
\qquad
\nabla^2 f(0)=
\begin{pmatrix}
2&0&0\\
0&2&0\\
0&0&2
\end{pmatrix}.
\]
Hence condition {\rm(i)} of Theorem~\ref{thm:SOSC} is automatic. For the chosen
direction $u=(1,1,0)$ and any
\[
w\in T^2_{0,u}X=T^{2,a}_{0,u}X,
\]
the second-order expression reduces to
\[
\langle u,\nabla^2 f(0)u\rangle
=4>0.
\]
Thus condition {\rm(ii)} is satisfied.

Finally, we discuss condition {\rm(iii)}. Near $0$, the set $X$ is the union of
two smooth analytic branches obtained by combining
\(
z=x^2
\)
with
\(
y=\pm x\sqrt{1+x}.
\)
Hence the parabolic regularity required in Theorem~\ref{thm:SOSC} can again be
verified branchwise. Therefore Theorem~\ref{thm:SOSC} applies, and $0$ is a
strict local minimizer of $f$ on $X$ with quadratic growth.
\end{example}

\begin{example}[A warning example]
\label{ex:opt-warning}
Consider the analytic hypersurface
\[
X=\{(x,y,z)\in\mathbb R^3:\ z^2-x^3y^3=0\},
\qquad
u=(1,-1,0)\in T_0X.
\]
As shown in Section \ref{subsec:difference},
\[
T_0X=\{(a,b,0)\in\mathbb R^3:\ a,b\in\mathbb R\},
\quad
T^2_{0,u}X=\emptyset,
\quad
T^{2,a}_{0,u}X=\mathbb R^3.
\]
Now consider the analytic objective function
\[
f(x,y,z)=x^2+y^2+z.
\]
Then
\[
\nabla f(0)=(0,0,1),\qquad
\nabla^2 f(0)=
\begin{pmatrix}
2&0&0\\
0&2&0\\
0&0&0
\end{pmatrix}.
\]
We claim that $0$ is a local minimizer of $f$ on $X$. Indeed, if
$(x,y,z)\in X$, then
\(
z^2=x^3y^3.
\)
Hence $xy\ge 0$, and therefore
\(
|z|=(xy)^{3/2}.
\)
By the arithmetic--geometric mean inequality,
\(
xy\le \frac{x^2+y^2}{2},
\)
so
\(
|z|\le \left(\frac{x^2+y^2}{2}\right)^{3/2}.
\)
Thus
\[
f(x,y,z)=x^2+y^2+z
\ge
x^2+y^2-|z|
\ge
x^2+y^2-\left(\frac{x^2+y^2}{2}\right)^{3/2}.
\]
If we set \(r^2=x^2+y^2\), then
\[
f(x,y,z)\ge r^2-\frac{1}{2^{3/2}}r^3
= r^2\left(1-\frac{r}{2^{3/2}}\right)\ge 0
\]
for all feasible points sufficiently close to $0$. Hence $0$ is a local
minimizer of $f$ on $X$.

However, if one incorrectly replaces the geometric second-order tangent set
$T^2_{0,u}X$ by the algebraic set $T^{2,a}_{0,u}X=\mathbb R^3$, then the
second-order expression becomes
\[
\langle u,\nabla^2 f(0)u\rangle+\langle \nabla f(0),w\rangle
=
4+w_3.
\]
Taking the infimum over \(w\in T^{2,a}_{0,u}X=\mathbb R^3\), one gets
\[
\inf_{w\in T^{2,a}_{0,u}X}
\Bigl(
\langle u,\nabla^2 f(0)u\rangle+\langle \nabla f(0),w\rangle
\Bigr)
=
\inf_{w_3\in\mathbb R}(4+w_3)
=
-\infty.
\]
Thus the algebraic replacement would falsely suggest failure of the
second-order necessary condition, even though $0$ is a local minimizer.

This example shows the inconvenience that occurs when
\[
T^2_{0,u}X\neq T^{2,a}_{0,u}X.
\]
The algebraic set may contain many second-order directions that are not
geometrically realizable, and using it in place of $T^2_{0,u}X$ may destroy the
correctness of the optimality conditions. Therefore, the equality
\[
T^2_{0,u}X=T^{2,a}_{0,u}X
\]
established in Section \ref{sec:equality} is essential if one wants to check second-order
optimality conditions algebraically.
\end{example}

\section*{Acknowledgement}
This research  is funded  by Vietnam Ministry of Education and Training (MOET) under grant number B2025-CTT-04.

\end{document}